\title{Primality of numbers of the form ${ap^{k}+1}$}
\author{Ariko Stephen Philemon\thanks{
                  Makerere University}}

\documentclass{article}
\usepackage{amssymb}
\usepackage{amsmath}
\usepackage{graphicx}
\usepackage{amsthm}

\newtheorem{theorem}{Theorem}[section]
\newtheorem{lemma}{Lemma}[section]
\newtheorem{conjecture}{Conjecture}[section]
\newtheorem{corollary}{Corollary}[theorem]

\newenvironment{keywords}{Keywords:}{}

\theoremstyle{definition}
\newtheorem*{definition}{Definition}

\theoremstyle{remark}
\newtheorem*{remark}{Remark}

\newtheorem{example}{Example}[section]

\usepackage{geometry}
\begin{document}

\newpage
\maketitle

\begin{abstract}
In 1876, Edouard Lucas showed that if an integer $b$ exists such that $b^{n-1} \equiv 1 (\mathrm{mod} \ n)$  and $b^{(n-1)/p} \not\equiv 1( \mathrm{mod} \ n)$ for all prime divisors $p$ of $n-1$ , then $n$  is prime, a result known as Lucas’s converse of Fermat’s little theorem. This result was considerably improved by Henry Pocklington in 1914 when he showed that it’s not necessary to know all the prime factors of $n-1$ in order to determine if $n$ is prime. In this paper we optimize Pocklington’s primality test for integers of the form $ap^{k}+1$ where $p$  is prime, $a<p$, $k\ge 1$. An extension of Lucas’s converse of Fermat’s little theorem is given. We also prove a new general-purpose primality test that requires that only a single odd prime divisor  of $n-1$ be found for the test to be implemented. Contrary to the well-known result: There are infinitely many Fermat pseudoprimes to any base $b$; In this paper we prove the finitude of Fermat pseudoprimes in some forms of integers.
\end{abstract}


\section{Introduction}
The problem of distinguishing primes from composite integers has been of interest to professional and amateur mathematicians alike for many centuries up to date. A number of primality tests have been established; Some of these tests such as Lucas’s converse of Fermat’s little theorem, Pocklington primality test, Proth’s test, Lucas Lehmer test among others determine whether a number is prime with absolute certainty while others such as Fermat’s Primality test, Miller-Rabin test report a number is composite or a probable prime. The previous tests depend on the factorization of $n-1$ or $n+1$ in order to determine the primality of $n$. More information on these tests can be found in \cite{R1}, \cite{R3}, \cite{R5}, \cite{R6}. In this paper we prove a relatively more efficient primality test for integers $n$ of the form $ap^{k}+1$, $a<p$, $k\ge1$ where $p$  is an odd prime. This test does not require computation of some greatest common divisors required in Pocklington’s primality test. Much effort is put in determining for which positive integers of this form the divisibility relation $p^{k} \ | \ \phi(n)$  holds from which the optimized test is deduced using properties of the order of an integer. From this result, a possible extension of Lucas’s converse of Fermat’s little theorem is discussed.
We also prove a new general-purpose primality test for any positive integer $n$ . Unlike Lucas’s converse of Fermat’s little theorem which provides no information on the prime factorization of $n$ , this paper shows that if a positive integer $n$  is a pseudoprime to the base $b$  with $p \ | \  n-1$  for some odd prime $p$ and $b^{(n-1)/p} \not\equiv 1( \mathrm{mod} \ n)$ then $n$ has a prime factor of the form $tp+1$. Although this new primality test has limited application, it’s still worthy attention because of its simplicity.

\begin{definition}
Let $b$  be a positive integer. If $n$ is a composite positive integer and $b^{n-1} \equiv 1( \mathrm{mod} \ n)$ , then $n$ is called a Fermat pseudoprime to the base $b$.
\end{definition}

It’s a well-known result that there are infinitely many Fermat pseudoprimes to any base $b$. In this paper we prove a simple and yet elegant result; For a given positive integer $a$, there are only finitely many Fermat pseudoprimes of the form $ap^k+1$, $k = 1, 2$, to any base $b>1$ where $p$ runs through the primes. We also show that this result holds for $k=3$, $a \not = m^3$. Substantial experimental data suggests this result holds for higher powers of $p$. That is, for a given integer $a$ and integer $k>3$, there are finitely many pseudoprimes of the form $ap^{k}+1$ to any base $b>1$. Nevertheless, it’s still possible that these pseudoprimes are infinitely many to any base $b>1$. 

\begin{definition}
Let $a$ and $n$ be relatively prime integers. The order of $a$ modulo $n$ denoted by $\mathrm{ord}_{n}a$ is the least positive integer $x$ such that $a^{x} \equiv 1 (\mathrm{mod} \ n)$.
\end{definition}
\begin{theorem} \label{Theorem 1.1}
Let $a$ and $n$ be relatively prime integers, then a positive integer $x$ is a solution of the congruence $a^{x} \equiv 1(\mathrm{mod} \ n)$  if and only if $\mathrm{ord}_{n}a \ | \ x$. In particular, $\mathrm{ord}_{n}a \ | \ \phi(n)$.
\end{theorem}

\begin{theorem}[Pocklington’s Primality Test] \label{Theorem 1.2}
Suppose that $n$ is a positive integer with $n-1=FR$ where $(F, R)=1$ and $F>R$. The integer $n$ is prime if there exists an integer $b$ such that $(b^{(n-1)/q}-1, \  n)$ whenever $q$ is a prime with $q \ | \ F$ and $b^{n-1 } \equiv 1(\mathrm{mod} \ n)$.
\end{theorem}

\section{Primes of the form $\boldsymbol{ap+1}$}
In this section, we prove a primality test for integers of the form $ap^k+1$ with$\ k=1$. Later we will generalize this test for higher powers of $p$.

\begin{lemma} \label{Lemma 2.1}
Let $n=ap+1$, where $a$ is a positive integer and $p$ is an odd prime. If $p\ |\ \phi(n)$ then $n=(sp+1)(tp+1)$ for some integer $sp+1$ and prime $tp+1$.
\end{lemma}

\begin{proof}
Let $n={\ p_1}^{a_1}{p_2}^{a_2}\ \cdots \ {p_k}^{a_k}$ be the prime power factorization of $n$. We have  $\phi(n)={p_1}^{a_1-1}\ (p_1-1){{\ p}_2}^{a_2-1}\ (p_2-1)\cdots \ {p_k}^{a_k-1}\ (p_k-1)$. The condition $p\ |\ \phi (n)$ implies  $p=\ p_i\ $or $p\ |\ p_i-1$ for some $i=1$, $2$, $\dots $,$\ k$.  If  $p=\ p_i$, then $p\ |\ n-ap=1$, which is not possible. Hence $p\ |\ p_i-1$ for some $i$. Thus, $p_i=tp+1$ for some integer $t$. Thus, $n=mp_i=m(tp+1)=ap+1.$ Factoring out $p$, we have  $p(a-mt)=m-1$,  $p\ |\ m-1$, $m=sp+1$ for some integer $s$. Thus, $n=mp_i=(sp+1)(tp+1)$. This completes the proof.
\end{proof}

\begin{remark}
If $n$ is odd and composite, we must have  $s\ge 2$, $t\ge 2$. It follows that for all $a<4(p+1)$, we have $p\ |\ \phi(n)$ if and only if $n$ is prime.
\end{remark}

\begin{theorem} \label{Theorem 2.1}
Let $n=ap+1$ where $a$ is even and $p$ is an odd prime with $a<4(p+1)$. If there exists an integer $b$ such that $b^{n-1}\equiv 1\ (\mathrm{mod}\ n)$ and $b^a\not\equiv 1(\mathrm{mod}\ n)$ then $n$ is prime.
\end{theorem}

\begin{proof}
We will show that if $n$ is composite and $b^{n-1}\equiv 1(\mathrm{mod}\ n)$ then $b^a\equiv 1(\mathrm{mod}\ n)$.   Assume $n$ is composite and $b^{n-1}\equiv 1(\mathrm{mod}\ n)$. From Theorem~\ref{Theorem 1.1}, $\mathrm{ord}_nb\ |\ \phi(n)$. Therefore if  $p\ |\ \mathrm{ord}_nb\ $we have $p\ |\ \phi(n)$ and from Lemma~\ref{Lemma 2.1} we know $n$ is prime, a contradiction because $n$ is assumed composite. Hence, we must have $p\nmid \mathrm{ord}_nb$, equivalently $(\mathrm{ord}_nb$, $p)=1$. From Theorem~\ref{Theorem 1.1}, we also note that  $\mathrm{ord}_nb\ |\ n-1=ap$.  Thus, $\mathrm{ord}_nb\ |\ a$ and from Theorem~\ref{Theorem 1.1}, $b^a\equiv 1\ (\mathrm{mod}\ n)$. Consequently if  $b^{n-1}\equiv 1(\mathrm{mod}\ n)$ and $b^a\not\equiv 1(\mathrm{mod}\ n)$, then we know $n$ is prime.
\end{proof}

\begin{remark}
A slightly more efficient primality test is obtained by replacing the hypothesis $b^{n-1}\equiv 1\ (\mathrm{mod}\ n)$ with $b^{(n-1)/2}\equiv \pm 1\ (\mathrm{mod}\ n)$.
\end{remark}

\begin{example}
Suppose we want to test whether $547=42\cdot 13+1$ is prime. Using fast modular exponentiation techniques, it can be verified that $2^{546/2}\equiv -1(\mathrm{mod}\ 547)$ and $2^{42}\equiv 475\not\equiv 1(\mathrm{mod}\ 547)$ and from Theorem 2.1, $547$ is prime. 
\end{example}
From Theorem~\ref{Theorem 1.2}, we note $n=b^a-1$ is the largest integer $n$ such that $b^a\equiv 1\ (\mathrm{mod}\ n)$, $b>1$. Assume $n>b^a-1$ or equivalently $p>(b^a-1)/a$ . It follows that if  $p>(b^a-1)/a$, then $b^a\not\equiv 1\ (\mathrm{mod}\ n)$. Furthermore if $b^{n-1}\equiv 1\ (\mathrm{mod}\ n)$ and $p<(a-1)/4$, from Theorem~\ref{Theorem 2.1} we know $n$ is prime. We state this result as a corollary

\begin{corollary} \label{Corollary 2.1}
Let $n=ap+1$ where $a$ is even and $p$ is an odd prime with $p>\{max(a-1)/4,\ {(b}^a-1)/a\ \}$. If $b>1$ is a positive integer and $b^{n-1}\equiv 1\ (\mathrm{mod}\ n)$ then $n$ is prime.
\end{corollary}

\begin{remark}
It's a well-known result that there are infinitely many Fermat pseudoprimes to any base $b$. However, Corollary~\ref{Corollary 2.1} demonstrates that for a given positive integer $a$, there are only finitely many Fermat pseudoprimes of the form $ap+1$ to any base $b>1$.
\end{remark}
\begin{example}
Taking $b=2$ and $a=2$, we compute $(b^a-2)/a=(2^2-2)/2=1$. Assume $p>2$, Corollary~\ref{Corollary 2.1} tells us that if $n=2p+1$ then  $2^{n-1}\equiv 1\ (\mathrm{mod}\ n)$ if and only if $n$ is prime or equivalently $p$ is a Sophie Germain prime if and only if $2^{2p}\equiv 1\ (\mathrm{mod}\ 2p+1)$. If we take $b=2$ and $a=6$, we have $(2^6-2)/6=31/3<11$. Taking $p\ge 11$ and $n=6p+1$, we have $2^{n-1}\equiv 1\ (\mathrm{mod}\ n)$ if and only if $n$ is prime.
\end{example}

Lemma~\ref{Lemma 2.1} can be extended to provide a primality test for any positive integer $n$ with $p\ |\ n-1$ for some odd prime $p$. From Lemma ~\ref{Lemma 2.1}, we have $p\ |\ \phi (n)$ if and only if $n=(sp+1)(tp+1)$ for some prime $tp+1$ and integer $sp+1$. To prove that $p\ |\ \phi (n)$ if and only if $n$ is prime, it suffices to show that $sp+1\nmid n$ for all $s\ge 1$, $sp+1\le \sqrt{n}$.

\begin{theorem}[General purpose primality test] \label{Theorem 2.2}
Let $n-1=mp$ for some odd prime $p$. If $sp+1\nmid n$ for all integers $1<sp+1\le \sqrt{n}$  and there exists an integer $b$ such that $b^{n-1}\equiv 1\ ( \mathrm{mod}\ n)$ and $b^m\not\equiv 1 ( \mathrm{mod} \ n)$ then $n$ is prime.
\end{theorem}

\begin{remark}
Because of the trial divisions involved, Theorem~\ref{Theorem 2.2} has limited application. It can only be implemented when the prime $p$ is large enough such that there are a few integers $sp+1$ to check. Theorem~\ref{Theorem 2.2} tells us that if $n$ is a pseudoprime to the base $b$ and $b^m\not\equiv 1(\mathrm{mod}\ n)$ then $n$ has a prime factor of the form $tp+1$. This new primality test is a combination of modular exponentiation and trial division.
\end{remark}

\begin{example}
To test $n=700001=100000\cdot 7+1$ for primality, we find that $2^{n-1}\equiv 1\ (\mathrm{mod}\ n)$ and $2^{(n-1)/7}\equiv 158306\not\equiv 1(\mathrm{mod}\ n)$. To complete the primality test, we verify that $7s+1\nmid n$ for all integers $7s+1\le \sqrt{n}=\sqrt{700001}<837\ $. Because $n$ is odd, only the odd integers of the arithmetic progression $7s+1$ may be considered. Note that $s\le 119$. We quickly find that $7s+1\nmid n$, $s=2,\ 4,\dots,\ 118$. Therefore, by Theorem~\ref{Theorem 2.2} $n$ is prime.
\end{example}

\section{Generalization of Theorem~\ref{Theorem 2.1} for higher powers of $\boldsymbol{p}$}
In this section we generalize the primality test presented in Theorem~\ref{Theorem 2.1} for higher powers of $p$. Using a similar argument presented in the proof of Lemma~\ref{Lemma 2.1}, it can be shown that if  $n=ap^k+1$, where $a$  and $k$ are positive integers, $p$ is a prime with $a<p$ then  $p^k\ |\ \phi (n)$ if and only if $n$ is prime. It follows that if $n$ is composite and $b^{n-1}\equiv 1(\mathrm{mod}\ n)$, the highest power of $p$ in $\mathrm{ord}_nb$ is less than $p^k$ so that $b^{{ap}^{k-1}}=b^{(n-1)/p}\equiv 1(\mathrm{mod}\ n)$. We proceed to give a detailed proof.

\begin{lemma} \label{Lemma 3.2}
Let $a<p$, with $p$ an odd prime and $k\ge 1$ such that with $n=ap^k+1$ we have $p^k\ |\ \phi(n) $. Then $n$ is prime.
\end{lemma}

\begin{proof}
Assume that $n$ is not prime. Let
\begin{equation*}
n=ap^k+1=A{q_1}^{b_1}\dots \ {q_j}^{b_j},
\end{equation*}
where $q_1,\dots ,\ q_j$ are primes such that $p\ |\ q_i-1$ for $i=1,\dots , j$ and $A$ consists of primes which are not congruent to $1$ modulo $p$ (if $A > 1)$. Since $p^k\ |\ \phi \left(n\right)$, we have $\phi \left(n\right)= p^{k}b$ for some $b<a$ (if $b=a$ then $\phi \left(n\right)=n-1$, so $n$ is prime which is false). If $b_i \ge 2$ for some $i$ then $q_i\ |\ \phi \left(n\right)$. Since $q_i\neq p$, we have $q_i\ |\ b$. Hence, $p <q_i \le b < a < p$, a contradiction. Thus, $b_1=\dots =b_j=1$. Since $p^kb\ =\phi (n)=\phi (A)(q_1 -1)\cdots(q_j -1)$, we can write $q_i =p^{u_i}s_i + 1$, where $s_i$ is coprime to $p$ and $u_1+\dots + u_j= k$.  Note that $s_1,\dots ,s_j$ are all even. Going back we get
\begin{equation}\label{eqn1}
p^k a+1= A(p^{u_1}s_1+1)\cdots (p^{u_j}s_j +1)
\end{equation}
Expanding the right-hand side we get
\begin{equation*}
 p^{k}As_1\cdots s_j + \mathrm{other\ terms}
\end{equation*}
From the above expansion, we see that $As_1\cdots s_j \le a<p$. On the other hand, reducing  $\eqref{eqn1}$ modulo $p$, we get $A \equiv 1 \ (\mathrm{mod} \ p)$. This forces $A=1$ since $A<p$. Because $n$ is assumed composite, we must have $j \ge 2$.  Assume $u_1  \le  \cdots \le u_j$. There's a positive integer $t$ such that $u_1=\dots = u_t=c<u_{t+1} \le \cdots \le u_j$. Reducing  $\eqref{eqn1}$ modulo $p^{c+1 }$, we get $ p^c s_1+ \cdots +  p^c s_t \equiv 0 \  (\mathrm{mod} \  p^{c+1}) $ (this holds because $j \ge 2$, $k>c$ ), so  $p \ | \ s_1 + \cdots + s_t$. Since $s_1, \dots , s_t$ are all even (in aparticular at least 2), one proves by induction that $s_1+ \cdots + s_t \le s_1 \cdots s_t$ with the base induction step being $s_1+s_2 \le s_1s_2$ which holds since its equivalent to $(s_1-1)(s_2-1) \ge 1.$  Hence, $ p \le s_1 + \cdots + s_t \le s_1 \cdots s_t  \le a<p$ , a contradiction. Therefore $n$ must be prime.
\end{proof}

\begin{theorem} \label{Theorem 3.1}
Let $n=ap^k+1$, $a$ and $k$ are positive integers, $p$ is an odd prime, $a<p$. If there exists an integer $b$ such that $b^{n-1}\equiv 1\ (\mathrm{mod}\ n)$ and $b^{(n-1)/p}\not\equiv 1(\mathrm{mod}\ n)$ then $n$ is prime.
\end{theorem}

\begin{proof}
Assume $n$ is composite and  $b^{n-1}\equiv 1\ (\mathrm{mod}\ n)$. From Theorem~\ref{Theorem 1.1},  $\mathrm{ord}_nb\ |\ n-1=ap^k$. Since  $(a$, $p^k)=1$, we have $\mathrm{ord}_nb=d_1d_2$, $d_1\ |\ a$,  $d_2\ |{\ p}^k$,  $d_2={\ p}^t$.  From Lemma~\ref{Lemma 3.2}, we must have  $0\le t\le k-1$ hence  $d_2\ |\ {\ p}^{k-1}$.  ${\mathrm{ord}_nb=d}_1d_2\ |\ a\cdot p^{k-1}\ $. It follows from Theorem 1.1 that $b^{(n-1)/p}=b^{ap^{k-1}}\equiv 1\ (\mathrm{mod}\ n)$. Consequently if $b^{n-1}\equiv 1\ (\mathrm{mod}\ n)$ and $b^{(n-1)/p}\not\equiv 1(\mathrm{mod}\ n)$ then we know $n$ is prime.
\end{proof}

\begin{example}
To test  $727=6\cdot {11}^2+1$ for primality; Using fast modular exponentiation, it can be shown that $2^{6\ \cdot {\ 11}^2}\equiv 1\ (\mathrm{mod}\ 727)$ and $2^{6\ \cdot \ 11}\equiv 590\not\equiv 1\ (\mathrm{mod}\ 727)$. Therefore, from Theorem~\ref{Theorem 3.1}, $727$ is prime. 
\end{example}
Alternatively, we can make use of Pocklington's primality test to show that $727$ is prime. The steps required to show $727$ is prime are exactly the same as in the optimized test except the additional gcd check required in Pocklington's test. i.e. there's need to further verify that $(590,\ 727)=1$.

\section{Generalization of Lemma~\ref{Lemma 3.2} for $\boldsymbol{am+1}$ integers}
Generalization of Lemma~\ref{Lemma 3.2} will provide a relatively more efficient primality test for a broader set of positive integers. Substantial experimental data suggests that if $n=am+1$, $a<p$, $p$ is the least prime divisor of $m$, then $m\ |\ \phi (n)$ if and only if $n$ is prime.

\begin{conjecture} \label{Conjecture 4.1}
Let $n=am+1$, where $a$ and $m$ are positive integers and let $p$ be the least prime divisor of $m$. If $a<p$ and  $m\ |\ \phi (n)$ then $n$ is prime.
\end{conjecture}
Conjecture 4.1 is a generalization of Lemma~\ref{Lemma 3.2}. Due to the difficulty in proving Lemma~\ref{Lemma 3.2}, a proof of this conjecture is expected to be exceptionally difficult, perhaps a completely new idea may be required to obtain its proof.  We challenge the reader to find a proof to this conjecture or a single counter example. The motivation for finding its proof is in Theorem~\ref{Theorem 4.2}. As will be shown shortly, Theorem~\ref{Theorem 4.2} is an extension of the well-known Lucas's converse of Fermat's little theorem. But first we shall prove  Theorem~\ref{Theorem 4.1}, a basis for the extended primality test.
The reader should find no difficulty proving the trivial case when $a=2$ and $m$ is odd. If $n=2m+1$, and $m\ |\ \phi (n)$, we have $\phi(n)=km$, $k\le 2$.  Because $\phi(n)$ is even, we must have $k=2$, $\phi(n)=2m=n-1$, thus $n$ is prime.

\begin{theorem} \label{Theorem 4.1}
Let $n=am+1$, where $a$ and $m>1$ are relatively prime positive integers such that $n$ is prime whenever  $m\ |\ \phi (n)$. If for each prime $q_i$ dividing $m$, there exists an integer $b_i$ such that ${b_i}^{n-1}\equiv 1\ (\mathrm{mod}\ n)$ and ${b_i}^{(n-1)/q_i}\not\equiv 1(\mathrm{mod}\ n)$ then $n$ is prime.
\end{theorem}

\begin{proof}
From Theorem~\ref{Theorem 1.1},  $\mathrm{ord}_nb_i\ |\ n-1$. Let $m={q_1}^{a_1}{q_2}^{a_2}\cdots {q_k}^{a_k}$ be the prime power factorization of $m$. Again, from Theorem~\ref{Theorem 1.1}, ${b_i}^{(n-1)/q_i}\not\equiv 1(\mathrm{mod}\ n)$ implies $\mathrm{ord}_nb_i\ \nmid (n-1)/q_i$. We leave to the reader to verify that the combination of $\mathrm{ord}_nb_i\ |\ n-1$ and $\mathrm{ord}_nb_i\ \nmid (n-1)/q_i$ implies ${q_i}^{a_i}\ |\ \mathrm{ord}_nb_i$. From Theorem~\ref{Theorem 1.1}, $\mathrm{ord}_nb_i\ |\ \phi (n)$  therefore for each $i$, ${q_i}^{a_i}\ |\ \phi (n)$ hence $m\ |\ \phi (n)$. Because $n$ is assumed prime whenever $m\ |\ \phi (n)$ , we conclude $n$ is prime.
\end{proof}
Theorem~\ref{Theorem 4.1} has little practical value on its own but becomes powerful when the integer $m$ is known beforehand. Assuming the truth of Conjecture~\ref{Conjecture 4.1}, Theorem~\ref{Theorem 4.1} is an optimized primality test for such integers. 

\begin{theorem}[Extended Lucas's converse of Fermat's little theorem] \label{Theorem 4.2}
Assume Conjecture~\ref{Conjecture 4.1} holds. Let $n=am+1$, where $a$ and $m$ are positive integers and let $p$ be the least prime divisor of $m$, $a<p$. If for each prime $q$ dividing $m$, there exists an integer $b$ such that $b^{n-1}\equiv 1\ (\mathrm{mod}\ n)$ and $b^{(n-1)/q}\not\equiv 1(\mathrm{mod}\ n)$ then $n$ is prime.
\end{theorem}

\begin{proof}
Assuming the truth of Conjecture~\ref{Conjecture 4.1}, Theorem~\ref{Theorem 4.2} follows directly from Theorem~\ref{Theorem 4.1}
\end{proof}
Taking $a=1$, $m=n-1$ and $p=2$, we obtain the well-known Lucas's converse of Fermat's little theorem. Theorem~\ref{Theorem 4.2} is thus an extension of Lucas's converse of Fermat's little theorem. Since the factorization of $a$ is not necessary, the larger the value of $a$, the faster the primality test. We illustrate Theorem~\ref{Theorem 4.2} with an example.

\begin{example}
Show that $n=7867=18 \cdot 19 \cdot 23+1$ is prime. Taking $a=18$ and $m=19 \cdot 23$; Using modular exponentiation, we find that $2^{n-1}\equiv 1\ (\mathrm{mod}\ n)$, $2^{(n-1)/19}\equiv 5437\not\equiv 1(\mathrm{mod}\ n)$ and $2^{(n-1)/23}\equiv 7369\not\equiv 1(\mathrm{mod}\ n)$. From Theorem~\ref{Theorem 4.2}, $n$ is prime. Using Pocklington's primality test, there's need to further verify that $(5437,\ n)=(7369,\ n)=1$.
\end{example}
A generalization of Theorem~\ref{Theorem 4.1} is possible in which the condition $(a,\ m)=1$ is dropped.

\begin{theorem} \label{Theorem 4.3}
Let $n-1=\prod^k_{i=1}{{p_i}^{s_i}}$ where $p_i$ are distinct primes, $s_i\ge 1$. Let $m=\prod^k_{i=1}{{p_i}^{t_i}}$, $0\le t_i\le s_i$ be an integer such that $n$ is prime whenever $m\ |\ \phi (n)$. If for each prime $p_i$ dividing $m$, there exists an integer $b_i$ such that ${b_i}^{n-1}\equiv 1\ (\mathrm{mod}\ n)$ and ${b_i}^{(n-1)/({p_i}^{s_i-t_i+1})}\not\equiv 1(\mathrm{mod}\ n)$ then $n$ is prime.
\end{theorem}

\begin{proof}
Using a similar argument as in the proof of Theorem~\ref{Theorem 4.1}, for each $i$, we have ${p_i}^{{s_i-(s}_i-t_i+1)+1}={p_i}^{t_i}\ |\ \phi (n)$ hence $m\ |\ \phi (n)$. By the hypotheses of Theorem~\ref{Theorem 4.3}, $n$ is prime.
\end{proof}

\begin{remark}
Theorem~\ref{Theorem 4.3} is a strengthening of Theorem~\ref{Theorem 4.1}; It does not require that $m$ and $(n-1)/m$ be relatively prime as required in Theorem~\ref{Theorem 4.1}. Taking $s_i=t_i$ for all $i$, we have Theorem~\ref{Theorem 4.1}. Theorem~\ref{Theorem 4.3} is relatively more efficient than Theorem~\ref{Theorem 4.1} when ${t_i<s}_i$ for some $i$. Theorem~\ref{Theorem 4.4} demonstrates this idea.
\end{remark}

\begin{lemma} \label{Lemma 4.1}
Let $n=ap^2+1$, $a<p$, $p$ is an odd prime. If $p\ |\ \phi (n)$, then $n$ is prime.
\end{lemma}

\begin{proof}
From Lemma~\ref{Lemma 2.1}, $n=(sp+1)(tp+1)$ for some integer $sp+1$ and prime $tp+1$. Thus, $n=stp^2+(s+t)p+1=ap^2+1$. We see that $p\ |\ s+t$, $p\le s+t$. Since $s$ and $t$ are at least 2, we have $s+t\le st$. This holds because its equivalent to $(s-1)(t-1)\ge 1$. Thus, $n=stp^2+(s+t)p+1>stp^2\ge {p\cdot p}^2=p^3$, a contradiction because $n<p^3$. Therefore $s=0$, $n=tp+1$. This completes the proof.
\end{proof}

\begin{theorem} \label{Theorem 4.4}
Let $n=ap^2+1$, $p>a$ is prime. If an integer $b$ exists such that $b^{n-1}\equiv 1\ (\mathrm{mod}\ n)$ and $b^a\not\equiv 1\ (\mathrm{mod}\ n)$ then $n$ is prime.
\end{theorem}

\begin{proof}
From Lemma~\ref{Lemma 4.1}, $p\ |\ \phi (n)$ if and only if $n$ is prime. Because the integer $b$ satisfies the hypothesis of Theorem~\ref{Theorem 4.3}, $n$ is prime.
\end{proof}

\begin{remark}
Similar to Theorem~\ref{Theorem 2.1}, Theorem~\ref{Theorem 4.4} allows us to deduce the simple unexpected result; For a given positive integer $a$, there are finitely many Fermat pseudoprimes of the form $ap^2+1$  to any base $b>1$. In light of Theorem~\ref{Theorem 4.4}, this is true because $b^a\not\equiv 1\ (\mathrm{mod}\ n)$ for all $n>b^a$. This in effect removes the need to further verify that $(b^{ap}-1,\ n)=1$ as required in the more versatile Pocklington's primality test.
\end{remark}
Note that Theorem~\ref{Theorem 4.3} is an open theorem; Just like Theorem~\ref{Theorem 4.1}, it can be used for any suitable choice of the positive integer $m$ for which $m\ |\ \phi (n)$ if and only if $n$ is prime. The reader is welcome to investigate this property further.

\begin{lemma} \label{Lemma 4.2}
Let $a$ be a positive integer that is not a perfect cube. The diophantine equation $(xz+1)(yz+1)={az}^3+1$ has no solutions in positive integers $x$, $y$, $z$ with $z>a^2+2a$.  
\end{lemma}

\begin{proof}
The equation $(xz+1)(yz+1)=az^3+1$ can be rewritten as $az^2-xyz-(x+y)=0$. We shall show that for any positive integer solution $(x,y,z)$, we have $z \le a^2+2a. \ $
Note that $z \ | \ x+y$, therefore $z \le x+y. \ $  If ($x, y, z)$ is a solution, then 
\begin{equation*} z = \frac{xy+\sqrt{x^2y^2+4a(x+y)}} {2a}
\end{equation*}
 In order for $z$ to be rational, the discriminant must be a perfect square. Therefore $w^2 = x^2y^2+4a(x+y)$. We see that $w > xy$ and $w \equiv xy \ ( \mathrm{mod} \ 2)$. We can write $w = xy + 2t$, $t > 0$.
Substituting $w$ above,  $(xy+2t)^2 = x^2y^2+4a(x+y)$. Expanding and simplifying, $txy - ax - ay +t^2 = 0$. Multiplying through by $t$ and factoring, $(tx-a)(ty-a)=a^2 - t^3$. We must have $t \le a-1$ otherwise $RHS<0$ and $LHS \ge 0$. Because $a$ is not a perfect cube, $a^2 - t^3 \not = 0$. The remainder of the proof utilizes the result: If $ab = c \  $ where $a,b, c \not = 0$ are integers then $a+b \le c+1$ if $c>0$ and $a+b \le -(c+1)$ if $c < 0$.
We now consider two cases;

Case $1: \ $  $ a^2 - t^3 >0 \ ;$
Using the result above on the factored equation, we have $(tx-a)+(ty-a) \le a^2 - t^3+1 \le a^2$.
Hence, $z \le x+y \le (a^2+2a)/t \le a^2+2a \\ $.

Case $2: \ $  $ a^2 - t^3 < 0 \ ;$
As in case $1$, we have $(tx-a)+(ty-a) \le t^3 - a^2-1 \ $,  $x+y \le t^2 - (a^2-2a+1)/t < t^2$.  Hence , $z \le x+y < t^2 \le (a-1)^2 < a^2 +2a$
\end{proof}

\begin{theorem}\label {Theorem 4.5}
Let $n=ap^3+1$, $a\neq m^3$, $p>a^2+2a$  is prime. If an integer $b$ exists such that $b^{n-1}\equiv 1\ (\mathrm{mod} \ n)$ and $b^a\not\equiv 1\ (\mathrm{mod}\ n)$ then $n$ is prime.
\end{theorem}

\begin{proof}
From Lemma~\ref{Lemma 2.1} and Lemma~\ref{Lemma 4.2}, we have $p\ |\ \phi (n)$ if and only if $n$ is prime. Because the integer $b$ satisfies the hypothesis of Theorem~\ref{Theorem 4.3}, we conclude that $n$ is prime.
\end{proof}

\begin{remark}
We immediately deduce that for any positive integer $a$ that is not a perfect cube, there are finitely many Fermat pseudoprimes of the form $ap^3+1$ to any base $b>1$. 
\end{remark}
Conclusion. In this paper we have proved that for every positive integer $a$, there are finitely many Fermat pseudoprimes of the form $ap^k+1$, $k= 1, 2$  to any base $b>1$. We have also shown that this results holds for $k = 3$, $a \not = m^3$.  What other forms of integers do we have finitely many pseudoprimes? For a given positive integer $a$ and integer $k>3$, are there finitely many pseudoprimes of the form $ap^k+1$ to a given base $b$?
\section*{Acknowledgement.}
I thank my former lecturer, Dr. Bamunoba Alex Samuel for the encouragement and discussions  in Number Theory. Am also grateful for the referee's helpful comments and suggestions.

College of Eng., Design, Art and Technology, Makerere university.		Email:  ariko@cedat.mak.ac.ug

\end{document}